\documentclass[10pt,a4paper]{amsart}
\usepackage[utf8]{inputenc}
\usepackage{amsmath}
\usepackage{amsfonts}
\usepackage{amssymb}
\usepackage{enumerate}
\usepackage{graphicx}
\usepackage{multirow}
\usepackage{hyperref}
\usepackage{amsthm}
\usepackage{hhline}
\usepackage{mathdots}
\usepackage{xcolor}
\usepackage{colortbl}
\usepackage{listings}
\usepackage{accents}

\usepackage{mathtools}
\usepackage[foot]{amsaddr}
\usepackage{subfigure}
\numberwithin{table}{section}
\usepackage{longtable}
\usepackage{tabularx}
\newcolumntype{C}[1]{>{\centering\arraybackslash}p{#1}}
\usepackage{booktabs}
\usepackage[english]{babel}

\calclayout
\newtheorem{prop}{Proposition}[section]
\newtheorem{theorem}[prop]{Theorem}
\newtheorem{lemma}[prop]{Lemma}\newtheorem{theoremA}{Theorem}

\theoremstyle{remark}
\newtheorem{rem}[prop]{Remark}
\theoremstyle{definition}

\newtheorem{condition}{Condition}

\DeclareMathOperator{\Aut}{Aut}

\DeclareMathOperator{\ord}{ord}

\DeclareMathOperator{\Irr}{Irr}

\DeclareMathOperator{\Inf}{Inf}

\DeclareMathOperator{\Gal}{Gal}
\DeclareMathOperator{\ab}{ab}

\DeclareMathOperator{\Ind}{Ind}

\DeclareMathOperator{\Syl}{Syl}

\newcommand{\G}{\mathbf{G}}
\newcommand{\GF}{\mathbf{G}^F}
\newcommand{\GFd}{{\mathbf{G}^*}^{F^*}}

\begin{document}
\title{On the inductive McKay--Navarro condition for $\mathsf{B}_2(2^f)$ and $\mathsf{G}_2(3^f)$}
\author{Birte Johansson}
\
\address{FB Mathematik, RPTU Kaiserslautern--Landau.}

\begin{abstract}
We verify the inductive McKay--Navarro condition for the groups $\mathsf{B}_2(2^f)$ and $\mathsf{G}_2(3^f)$ and all primes if $f$ is odd. Further, we show that the equivariance part of the inductive condition holds for all integers $f$. 
\end{abstract}

\maketitle

\section{Introduction}
The McKay--Navarro conjecture is one of the so-called local-global conjectures that relate the representation theory of a finite group with that of some of its subgroups. In order to state it, we write $\Irr_{\ell'}(G)$ to denote the set of irreducible characters of a finite group $G$ with degree not divisible by a prime $\ell$. The McKay conjecture claims that $\Irr_{\ell'}(G)$ is in bijection with $\Irr_{\ell'}(N_G(R))$ where $R$ is a Sylow $\ell$-subgroup of $G$. 
Navarro refined this conjecture in \cite{navarro2004mckayrefinement} and claimed that this bijection can be chosen equivariant under the action of certain Galois automorphisms on the respective characters. This is known as the McKay--Navarro or Galois--McKay conjecture. 

In \cite{navarro2019reduction}, Navarro, Späth, and Vallejo proved a reduction theorem for the McKay--Navarro conjecture, yielding an inductive condition that should be verified for all finite simple groups. So far, this has been done for groups of Lie type in their defining characteristic in \cite{ruhstorfer2017navarro} and \cite{johansson2020} as well as for the Suzuki and Ree groups and all primes in \cite{johansson2021}. Ruhstorfer and Schaeffer Fry further showed in \cite{ruhstorferSF2022} that the inductive condition holds for all finite simple groups and the prime $\ell=2$ and thereby proved that the conjecture itself is true for $\ell=2$.

Parts of the inductive condition are concerned with the extensions of characters to automorphism groups and the action of group and Galois automorphisms on these extensions. The character extensions for the groups that have been considered so far are mostly invariant under this action, especially for groups of Lie type without diagonal automorphisms. However, we noticed that some characters of the group $3.\mathsf{G}_2(3)$ do not allow such an invariant extension for $\ell=3$. While one of these characters arises from the exceptional Schur multiplier of $\mathsf{G}_2(3)$, another one is a unipotent character that arises for all finite simple groups of type $\mathsf{G}_2$ over a field of characteristic $3$.  
This was the motivation to study the inductive McKay--Navarro condition for these groups. We show the following theorem.

\begin{theoremA} \label{theoremA}
The equivariance part of the inductive McKay--Navarro condition holds for the groups $\mathsf{B}_2(2^f)$, $\mathsf{G}_2(3^f)$ and any prime $\ell$ where $f \geq 1$. If $f$ is odd, the full inductive condition is satisfied.
\end{theoremA}

\section{Preliminaries}
\subsection{Inductive McKay--Navarro condition} Let $G$ be a finite group and $\ell$ a prime dividing $|G|$. The group $\mathcal{G}:=\Gal(\mathbb{Q}^{\ab}/\mathbb{Q})$ acts on the set $\Irr(G)$ of ordinary irreducible characters of $G$ by first evaluating the character and then applying the Galois automorphisms to the character values. We are interested in the subgroup $\mathcal{H} \subseteq \mathcal{G}$ consisting of those Galois automorphisms that map every root of unity $\zeta$ with $\ord(\zeta)$ prime to $\ell$ to $\zeta^{\ell^k}$ for some integer $k$. 

We denote the automorphism group of $G$ by $\Aut(G)$ and write $\Aut(G)_H$ for the setwise stabilizer of $H \leq G$ in $\Aut(G)$. An automorphisms $a$ of $G$ acts on a character $\psi$ of $G$ by $\psi^a(x)=\psi(x^{a^{-1}})$ for all $x \in G$. If $a$ is an inner automorphism, we also use the group element itself to denote conjugation with $g \in G$. We write $A_\psi$ to denote the elements of $A \leq \Aut(G)$ stabilizing the character $\psi$ of $G$.

We consider the following simplified condition.
\begin{condition} \label{cond}
Let $S$ be a non-abelian simple group, $G$ its universal covering group, $\ell$ a prime, and $R \in \Syl_\ell(G)$. 
\begin{enumerate}[(1)]
\item \textit{(Equivariance condition)} There exists an $\Aut(G)_R$-stable subgroup $N_G(R) \subseteq N \subsetneq G$ and an $\Aut(G)_R \times \mathcal{H}$-equivariant bijection $\Omega: {\Irr}_{\ell'}(G) \rightarrow {\Irr}_{\ell'}(N)$ preserving central characters.
\item \textit{(Extension condition)} For all $\psi \in {\Irr}_{\ell'}(G)$, there are $(\Aut(G)_R \times \mathcal{H})_\psi$-invariant extensions of $\psi$ and $\Omega(\psi)$ to $G \rtimes (\Aut(G)_R)_\psi$ and $N \rtimes (\Aut(G)_R)_\psi$, respectively.
\end{enumerate}
\end{condition}
This condition implies the inductive McKay--Navarro condition from \cite[Definition 3.5]{navarro2019reduction} for groups with cyclic outer automorphism groups since, in this case, we do not have to keep track of the scalars corresponding to the extensions on $C_{G \rtimes (\Aut(G)_R)_\psi}(N)$, see \cite[Lemma 4.6]{johanssondiss}.

\subsection{Action of $\mathcal{H}$ on roots of unity}
Before we start to work on the actual condition, we give an easy number-theoretical lemma where we collect some statements that we need in the following.
\begin{lemma}\label{lemmanumbertheory} Let $\ell$ be an odd prime and $f$ a positive integer.
\begin{enumerate}[(a)]
\item If $\ell \mid 3^f-1$ and $f$ is odd, then $\ell \equiv 1,11 \mod 12$. In particular, $\mathcal{H}$ then fixes $\sqrt{3}$.
\item If $\ell \mid 3^f+1$ and $f$ is odd, then $\ell \equiv 1,7  \mod 12$. In particular, $\mathcal{H}$ then fixes $\sqrt{-3}$.
\item If $\ell \mid q^2\pm q+1$ for $q:=3^f$, then $\ell \equiv 1  \mod 3$. In particular, $\mathcal{H}$ then fixes $\sqrt{-3}$.
\item If $\ell \mid 2^f-1$ and $f$ is odd, then $\ell \equiv 1,7 \mod 8$. In particular, $\mathcal{H}$ then fixes $\sqrt{2}$.
\item If $\ell \mid 2^{2f}+1$, then $\ell \equiv 1 \mod 4$.
\end{enumerate}
\end{lemma}
\begin{proof}
By Fermat's little theorem, we have $3^{\ell-1}=1 \mod \ell$. Further, we know by quadratic reciprocity that $3^{(\ell-1)/2} \equiv 1 \mod \ell $ if $\ell=1,11 \mod 12$ and  $3^{(\ell-1)/2} \equiv -1 \mod \ell $ if $\ell=5,7 \mod 12$.

In (a), we already know $3^{(\ell-1)/2}=1 \mod \ell$ since $f$ is odd. The second claim then follows from $$\sqrt{3}=\zeta_{12}-\zeta_{12}^5=\zeta_{12}^{11}-\zeta_{12}^7$$ for a primitive $12$-th root of unity $\zeta_{12}$.

In (b),  we have $3^f \equiv -1 \mod \ell$ and $3^{2(f, (\ell-1)/2)} \equiv 1$. Assume that $\ell \equiv 3 \mod 4$. Then $(\ell-1)/2$ is odd and we are therefore in the second case above which implies $\ell= 5,7 \mod 12$. If we have $\ell \equiv 1 \mod 4$, then we know $3^{2(f, (\ell-1)/4)} \equiv 1$ and we are therefore in the first case. This implies $\ell \equiv 1,7  \mod 12$. The second claim follows from $\sqrt{-3}=-\zeta_{3}+\zeta_{3}^2 $ for a third root of unity $\zeta_{3}$.

To prove (c), first assume that $\ell \mid q^2+q+1$. Since $(q^2+q+1)(q-1)=q^3-1$, we have $3^{(3f, \ell-1)} \equiv 1 \mod \ell$. If $ \ell \equiv 2 \mod 3$, this implies $3^{(f, \ell-1)} \equiv 1 \mod \ell$ and thereby $\ell \mid q-1$. This is not possible since it implies $\ell \mid (q^2+q+1)-(q-1)^2=3q$ and it follows $\ell \equiv 1 \mod 3$.

For $\ell \mid q^2-q+1$, we have similarly $3^{(6f, \ell-1)} \equiv 1 \mod \ell$. If $ \ell \equiv 2 \mod 3$, this implies $3^{(2f, \ell-1)} \equiv 1 \mod \ell$ and thereby $\ell \mid q+1$ or $\ell \mid q-1$. This is not possible since it implies $\ell \mid (q^2-1)-(q^2-q+1)=q-2$ which is not possible and it follows $\ell \equiv 1 \mod 3$.

Finally, (d) can be proved analogously to (a) and (e) follows directly from the second supplement to quadratic reciprocity.
\end{proof}

\subsection{Notions from Lusztig theory}
Let $\G$ be a connected reductive group defined over an algebraically closed field of positive characteristic and $F$ a Frobenius endomorphism of $\G$. If $\mathbf{L}$ is an $F$-stable Levi subgroup of $\G$, we denote the relative Weyl group of $\mathbf{L}$ in $\G$ by $\mathbf{W}_{\G}(\mathbf{L}):=N_{\G}(\mathbf{L})/\mathbf{L}$. Let $(\G^*, F^*)$ be in duality with $(\G, F)$ and $s \in \GFd$ semisimple. We write $\mathcal{E}(\GF, \mathbf{T}^F, s)$ for the set of irreducible consituents of the Deligne--Lusztig generalized character $R_{\mathbf{T}}^{\G}(s)$ where $\mathbf{T}\subseteq \G^*$ is an $F$-stable maximal torus containing $s$. 
The irreducible characters of $\GF$ are parametrized by pairs $(s, \nu)$ where $s \in \GFd$ is semisimple and $\nu$ is a unipotent character of $C_{\GFd}(s)$. If $\chi \in \Irr(\GF)$ has Jordan decomposition $(s, \nu)$, we also denote it by $\chi_{s, \nu}^{\G}$. See \cite[Chapter 2]{geckmalle2020} for details on these notions.
 
\section{Equivariance condition for type $\mathsf{G}_2$} 
In this section, we verify the equivariance part of Condition \ref{cond} for the groups $\mathsf{G}_2(3^f)$.
\subsection{Setting and Sylow twists} \label{sectiong2}

From now on, let $\G$ be a simple algebraic group of type $\mathsf{G}_2$ over an algebraically closed field of characteristic $3$. We denote the standard Frobenius map of $\G$ by $F_3$ and set $F:=F_3^f$, $q:=3^f$ for some integer $f \geq 2$. Then $G:=\GF=\mathsf{G}_2(3^f)$ is a simple group with trivial Schur multiplier. The outer automorphism group of $G$ is $D=\langle \gamma \rangle$ where $\gamma$ is the exceptional graph automorphism of $\G$. Therefore, we have $\gamma^2=F_3$ and $|D|=2f$.
Note that $\G$ has trivial center and $(\G,F)$ is self-dual. 

We consider the inductive McKay--Navarro condition for a prime $\ell$ dividing 
$$|G|=q^6 (q-1)^2(q+1)^2(q^2+q+1)(q^2-q+1)=q^6 \phi_1^2 \phi_2^2 \phi_3 \phi_6$$
where $\phi_d$ denotes the $d$-th cyclotomic polynomial evaluated at $q$. We assume that $\ell \geq 5$ since $\ell=2$ and $\ell=3$ have already been considered in \cite{johansson2021} and \cite{johansson2020}. 

Let $d$ be the order of $q$ modulo $\ell$. Then we have $\ell \mid \phi_d$. Let $\mathbf{S}_d$ be a Sylow $d$-torus of $\G$ and $N:=N_{G}(\mathbf{S}_d)$. 
Note that $\mathbf{T}_0:=\mathbf{S}_1$ is maximally split and $\mathbf{S}_d=C_{\G}(\mathbf{S}_d)$ is a maximal torus of $G$ since $d$ is a regular number for $\G$ as in \cite{springer74}. We write $\mathbf{N}:=N_{\G}(\mathbf{T}_0)$ and denote the relative Weyl groups by $W_d:=\mathbf{W}_{\G}(\mathbf{S}_d)$. Note that $W_1 \cong W_2$ is isomorphic to the dihedral group $D_{12}$ of order $12$ and $F$ acts trivially on it. Further, we have $W_3 \cong W_6 \cong C_6$, see \cite[Table 1]{kleidman1988maxsubgroups}.

By \cite[Theorem 5.19]{malle2008height0}, $N$ contains the normalizer of a Sylow $\ell$-subgroup and we can therefore consider the group $N$ as the local group in the inductive condition.

In the following, we work with so-called Sylow twists. We only introduce the notation and constructions that we need in the following and refer to \cite{spaeth2009mckayexceptional} for the details. 

Let $V:=\langle n_{\alpha}(\pm 1) \mid \alpha \in \Phi \rangle \leq \mathbf{N}$ be the extended Weyl group where $\Phi$ is the root system of $\G$ with respect to $\mathbf{T}_0$ and the $n_\alpha$ are as in \cite[Theorem 1.12.1]{gorensteinlyonssolomon1998}. We set $H:=V \cap \mathbf{T}_0$. Let $v_d \in V$ be a very good Sylow $d$-twist of $(\G,F)$ as chosen in \cite[Proof of Proposition 6.3]{spaeth2009mckayexceptional}. Note that we have $v_1=v_3^3=1$ and $v_2=v_6^3 \in Z(W_1)$. We often also write $v=v_d$ and keep in mind that $v$ depends on $d$.
Then, $v$ induces an isomorphism $G \cong \G^{vF}$ and $\mathbf{T}_0$ is a Sylow $d$-torus of $(\G, vF)$. Therefore, we can identify $\mathbf{T}_0^{vF}$ with $\mathbf{S}_d^F$ and $\mathbf{N}^{vF}$ with $N$.

\subsection{Character parametrization} \label{sectioneq}
As described in \cite{malle2008height0} and \cite[Section 4]{cabanesspaeth2013}, the characters in $\Irr_{\ell'}(\G^{vF})$ and $\Irr_{\ell'}(\mathbf{N}^{vF})$ can be parametrized by the set
$$\mathcal{M}=\{(s, \eta) \mid s \in \mathbf{T}_0^{vF} \text{ semisimple up to $\mathbf{N}^{vF}$-conjugacy, } \eta \in \Irr_{\ell'}(\mathbf{W}_{C_{\G}(s)}(\mathbf{T}_0)^{vF})\}$$
and there are $D$-equivariant bijections given by 
$$\psi_{\text{glob}}:\mathcal{M} \rightarrow \Irr_{\ell'}(\G^{vF}), \quad (s, \eta) \mapsto \chi^{\G}_{s, \mathcal{I}(\eta)},$$
$$\psi_{\text{loc}}:\mathcal{M} \rightarrow \Irr_{\ell'}(\mathbf{N}^{vF}), \quad (s, \eta) \mapsto \Ind_{(\mathbf{N}^{vF})_{\chi^{\mathbf{T}_0}_{s, 1}}}^{\mathbf{N}^{vF}}(\Lambda(\chi^{\mathbf{T}_0}_{s, 1})\eta).$$
Here, $\Lambda$ is an extension map for $\mathbf{T}_0^{vF} \unlhd \mathbf{N}^{vF}$ as in \cite[Definition 2.9]{cabanesspaeth2013}, and 
$\mathcal{I}: \Irr(\mathbf{W}_{C_{\G}(s)}(\mathbf{T}_0)^{vF}) \rightarrow \mathcal{E}(C_{\G}(s)^{vF}, \mathbf{T}_0^{vF}, 1)$ is a $D$-equivariant bijection, see \cite[Theorem 3.4]{cabanesspaeth2013}. 

Note that these parametrizations do not only exist for $\mathbf{T}_0$ but for every Sylow $d$-torus of $(\G,vF)$. Similarly, one can construct analogous maps for $(\G,F)$ and its Sylow tori. 

As we already see from these bijections, it will be important to know about the structure of $C_{\G}(s)$ for $s \in \mathbf{T}_0^{vF}$. 
These correspond to the centralizers $\mathbf{H}:=C_{\G}(s)$ of semisimple elements $s \in \mathbf{S}_d^F$ for $(\G,F)$.
Information about these centralizers can be found in \cite{luebeckhomepagecent} and \cite{enomoto1976g23}. Since $\mathbf{S}_d \subseteq \mathbf{H}$, we only have to consider semisimple elements with $\phi_d$ not dividing $[\GF:\mathbf{H}^F]$. We collect some of the data that we will use later in Table \ref{tabless}. The structure of the relative Weyl group can be read off from its order and the subgroups of $W_d$. 

\begin{table}
\begin{tabular}{cccccc}
\hline $d$ & Conjugacy Class & Type of $\mathbf{H}$& $|\mathbf{H}^F|$&$|s^N|$ &$\mathbf{W}_{\mathbf{H}}(\mathbf{S}_d)^F$ \\ \hline
$1$ & $A_1 =1 $ & $\mathsf{G}_2(q)$ & $q^6(q^2-1)(q^6-1)$ & $1$ & $D_{12}$ \\ 
 & $B_1 $ & $\mathsf{A}_1(q)\cdot \mathsf{A}_1(q)$ & $q^2(q^2-1)^2$ & $3$ & $C_2 \times C_2$ \\
& $C_{11}(i) $ & $\mathsf{A}_1(q)\cdot T(\phi_1)$ & $q(q-1)^2(q+1)$ & $6$ & $C_2 $ \\
&$C_{21}(i) $ & $\mathsf{A}_1(q) \cdot T(\phi_1)$ & $q(q-1)^2(q+1)$ & $6$ & $C_2 $ \\
&$E_{1}(i,j) $ & $\mathsf{A}_0(q) \cdot T(\phi_1^2)$ & $(q-1)^2$ & $12$ & $1$ \\ \hline 

$2$ &$A_1 =1 $ & $\mathsf{G}_2(q)$ & $q^6(q^2-1)(q^6-1)$ & $1$ & $D_{12}$ \\
&$B_1 $ & $\mathsf{A}_1(q) \cdot \mathsf{A}_1(q)$ & $q^2(q^2-1)^2$ & $3$ & $C_2 \times C_2$ \\
&$D_{11}(i) $ & $\mathsf{A}_1(q) \cdot T(\phi_2)$ & $q(q+1)^2(q-1)$ & $6$ & $C_2 $ \\
&$D_{21}(i) $ & $\mathsf{A}_1(q) \cdot T(\phi_2)$ & $q(q+1)^2(q-1)$ & $6$ & $C_2 $ \\
&$E_{4}(i,j) $ & $\mathsf{A}_0(q) \cdot T(\phi_2^2)$ & $(q+1)^2$ & $12$ & $1$ \\ \hline 

$3$ &$A_1=1 $ & $\mathsf{G}_2(q)$ & $q^6(q^2-1)(q^6-1)$ & $1$ & $C_{6}$ \\
&$E_{5}(i) $ & $\mathsf{A}_0(q) \cdot T(\phi_3)$ & $q^2+q+1$ & $6$ & $1$ \\ \hline

$6$ &$A_1=1 $ & $\mathsf{G}_2(q)$ & $q^6(q^2-1)(q^6-1)$ & $1$ & $C_{6}$ \\
&$E_{6}(i) $ & $\mathsf{A}_0(q) \cdot T(\phi_6)$ & $q^2-q+1$ & $6$ & $1$ \\ \hline

\end{tabular}
\caption{Conjugacy classes of semisimple elements $s \in \mathbf{S}_d^F$ using the notation of \cite[Table VII-1]{enomoto1976g23} and description of the structure of their centralizers $\mathbf{H}=C_{\G}(s)$ as given in \cite{luebeckhomepagecent}.}
\label{tabless}
\end{table}

We now want to show that the maps $\psi_{\text{loc}}$ and $\psi_{\text{glob}}$ can be chosen equivariant under the action of $\mathbf{N}^{vF} D \times \mathcal{H}$.

\begin{lemma} \label{lemmaextv}
There is an $\mathbf{N}^{vF} D \times \mathcal{G}$-equivariant extension map $\Lambda$ for $\mathbf{T}_0^{vF} \triangleleft \mathbf{N}^{vF}$.
\end{lemma}
\begin{proof}
If $d \in \{1,2\}$, this has already been shown in \cite[Corollary 4.4]{fry2020galoisequivariant}. We assume $d \in \{3,6\}$ and follow \cite[Proof of Lemma 4.1]{fry2020galoisequivariant}.

We first want to show that there is an $H^{vF} D \times \mathcal{G}$-equivariant extension map for $H^{vF}\triangleleft V^{vF}$. Recall that we have $v_1=v_3^3=1$  and $v_2=v_6^3$ as well as $H=H^{v_2F}$ and $V=V^{v_2F}$. We follow the construction from \cite[Proof of Lemma 4.6]{spaeth2009mckayexceptional}. As described there, we have
$H=H^{vF} \times H'$ for a group $H':=\langle x^{-1}x^v, x^{-1}x^{v^2} \mid x \in H \rangle$.

Since this is a direct product, we can extend every $\delta \in \Irr(H^{vF})$ trivially to $\widetilde{\delta} \in \Irr(H)$ with $\widetilde{\delta}|_{H'}=1$. 
We already know that there is an $H D \times \mathcal{G}$-equivariant extension map for $H \triangleleft V$ by \cite[Proof of Lemma 4.1]{fry2020galoisequivariant}. Let $\widehat{\delta} \in \Irr(V_{\widetilde{\delta}})$ be such an extension of $\widetilde{\delta}$. Since we have $V_{\widetilde{\delta}} \geq V^{vF}_\delta$, we obtain an extension $\widehat{\delta}|_{V^{vF}_\delta} \in \Irr({V^{vF}_\delta})$ of $\delta$.
By construction, all extension steps were $H D \times \mathcal{G}$-equivariant and therefore also $H^{vF} D \times \mathcal{G}$-equivariant. 

We can now construct an $\mathbf{N}^{vF} D \times \mathcal{G}$-equivariant extension map for $\mathbf{T}_0^{vF} \triangleleft \mathbf{N}^{vF}$ as in \cite[Lemma 4.2]{spaeth2009mckayexceptional} and \cite[Proof of Lemma 4.1]{fry2020galoisequivariant}.
\end{proof}

We now prove the equivariance part of Condition \ref{cond} by working in the twisted setting. Note that for $d=1$ this has already been shown in \cite{fry2020galoisequivariant}.
\begin{prop} \label{propeq}
There is a $D \times \mathcal{H}$-equivariant bijection $\Irr_{\ell'}(\mathbf{G}^{vF}) \rightarrow \Irr_{\ell'}(\mathbf{N}^{vF})$. 
\end{prop}
\begin{proof}
It is enough to show that the parametrizations $\psi_{\text{loc}}$ and $\psi_{\text{glob}}$ are $D \times \mathcal{H}$-equivariant.
For the global characters, we already know that the Jordan decomposition for $(\G,vF)$ is Galois-equivariant by \cite{srinivasanvinroot2019}. It remains to show that the $D$-equivariant bijection between $ \Irr_{\ell'}(\mathbf{W}_{C_{\G}(s)}(\mathbf{T}_0)^{vF})$ and $\mathcal{E}(\G^{vF}, \mathbf{T}_0^{vF}, 1)$ is $\mathcal{H}$-equivariant. For $d=2$, this follows from the fact the occurring relative Weyl groups only have rational characters and the unipotent $\ell'$-characters of groups of type $\mathsf{A}_1$ and $\mathsf{G}_2$ are also rational. For $d \in \{3,6\}$, the values of the irreducible characters of $\mathbf{W}_{\G}(\mathbf{T}_0)^{vF} \cong C_6$ and of the unipotent $\ell'$-characters of $\GF$ are contained in $\mathbb{Q}(\zeta_3)$ where $\zeta_3$ is a third root of unity. By Lemma \ref{lemmanumbertheory}(c), they are $\mathcal{H}$-invariant which shows the claim. 

For the local characters, we know that all $\eta \in \Irr_{\ell'}(\mathbf{W}_{C_{\G}(s)}(\mathbf{T}_0)^{vF})$ for semisimple $s \in \mathbf{T}_0^{vF}$ are $\mathcal{H}$-invariant. Moreover, we know from Lemma \ref{lemmaextv} that $\Lambda$ can be chosen $\mathbf{N}^{vF}D \times \mathcal{H}$-equivariant. This shows the claim.
\end{proof}

\section{Character extensions for type $\mathsf{G}_2$} \label{sectionext}
We now consider extensions of the local and global characters to $N \rtimes \Aut(G)_R$ and $G \rtimes \Aut(G)_R$, respectively, for some $R \in \Syl_\ell(N)$. 

If we have $d=1$, then we can assume that the groups $N$ and $\mathbf{T}_0^F$ are stable under the action of $D$. Since $\mathbf{T}_0^F$ is abelian, the same holds for its Sylow $\ell$-subgroup. 

Note that the Weyl group of $\G$ is generated by the reflections corresponding to the two simple roots $\alpha, \beta \in \Phi$, i.e.
$W_1=\langle s_{\alpha}, s_{\beta} \mid (s_\alpha s_\beta)^6=1, s_\alpha^2=s_\beta^2=1 \rangle .$
We further have $W_d \cong W_1^{vF}$.

From now on, we assume that $f$ is odd.
We first study extensions of the local characters.
\begin{lemma} \label{lemmad1}
Assume $d=1$. Then, every $\psi \in \Irr_{\ell'}(N)$ has an $(\Aut(G)_R \times \mathcal{H})_\psi$-invariant extension to $N \rtimes (\Aut(G)_R)_\psi$.
\end{lemma}
\begin{proof}
Let $F_0$ be a generator of $D_\psi$ and $\psi=\psi_{\text{loc}}(s,\eta)$ for a pair $(s, \eta) \in \mathcal{M}$. 

We first consider the case $s \neq 1$. As we can see in Table \ref{tabless}, $\eta$ is a linear rational character. Since $\chi^{\mathbf{T}_0}_{s, 1}$ is also linear, $\psi_0:=\Lambda(\chi^{\mathbf{T}_0}_{s, 1})\eta$ is a linear character. By Clifford theory, there is an $n \in N$ such that $\psi_0^{F_0 n}=\psi_0$. Thus, we can extend $\psi_0$ trivially to a $(D\times \mathcal{H})_\psi$-invariant character of $N_\psi \rtimes \langle F_0 n \rangle$. Induction to $N\rtimes \langle F_0 n \rangle$ and the canonical extension to the remaining inner automorphisms yields a character as claimed. 

Assume now that $s=1$, i.e. $\psi=\Inf_{W_1}^N(\eta)$ for the corresponding $\eta \in \Irr(W_1)$. Four of the irreducible characters of $D_{12} \cong W_1$ are linear and therefore correspond to linear characters $\psi$. These can be trivially extended to $(D\times \mathcal{H})$-invariant characters of $N \rtimes \langle F_0 \rangle$. 
The only remaining characters are $\eta_5, \eta_6 \in \Irr(W_1)$ of degree $2$ that are both $\gamma$-invariant. We know that $\gamma^2=F_3$ acts trivially on $W_1$. Therefore, $\eta_5$ and $\eta_6$ can be trivially extended to $W_1 \rtimes \langle \gamma^2 \rangle =W_1 \times \langle \gamma^2 \rangle$ and this extension is clearly $(D \times \mathcal{H})_\psi$-invariant.

On the other hand, we know that $\gamma$ and thereby also $\gamma^f$ acts on $W_1$ by interchanging the reflections $s_\alpha, s_\beta \in W_1$. Thus, we can use GAP \cite{GAP4} to construct the automorphism $\varphi$ of $W_1$ induced by $\gamma^f$ explicitly, compute the irreducible characters of $W_1 \rtimes \langle \varphi \rangle$, and thereby determine extensions of $\eta_5$ and $\eta_6$, respectively, see Table \ref{tabWchar}. As it turns out, one of these extended characters is rational and the values of the other extended character are contained in $\mathbb{Q}(\sqrt{3})$. By Lemma \ref{lemmanumbertheory}(a), both character extensions are $\mathcal{H}$-invariant. 

We now find a unique common extension of the characters of $W_1 \rtimes \langle \gamma^2 \rangle$ and $W_1 \rtimes \langle \gamma^f \rangle \cong W_1 \rtimes \langle \varphi \rangle$ extending $\eta_5$ and $\eta_6$ to $W_1 \rtimes \langle \gamma \rangle$. This extension is again invariant under $\mathcal{H}$ and inflation yields an $\mathcal{H}$-invariant extension of $\psi$ to $N \rtimes D_\psi$. Since $\psi$ is rational, the extension is also invariant under $(D \times \mathcal{H})_\psi=D_\psi \times\mathcal{H}$. Extending this to the inner automorphisms in $\Aut(G)_R$ yields the claim. 
\end{proof}

With a bit more work, we can prove the analogous statement for other values of $d$.

\begin{lemma} \label{lemmad2}
Assume $d=2$. Then, every $\psi \in \Irr_{\ell'}(N)$ has an $(\Aut(G)_R \times \mathcal{H})_\psi$-invariant extension to $N \rtimes (\Aut(G)_R)_\psi$.
\end{lemma}
\begin{proof}
We work in the twisted setting and consider $\mathbf{N}^{vF} \cong N$. We know from \cite[Section 3.A]{mallespaeth2015mckayfor2} that we find an automorphism $\gamma_2 \in \Aut(\G)$ stabilizing $\mathbf{T}_0$ that acts on $G$ in the same way as $\gamma$ such that $\gamma_2^{2f}=vF$. 
Let $\psi$ be the character corresponding to the pair $(s, \eta) \in \mathcal{M}$. If $s \neq 1$, we can argue as for $d=1$. 

If $s=1$, the irreducible characters of $W_1=W_1^{vF}\cong D_{12}$ are as we have seen rational and consist of four linear characters and two characters $\eta_5, \eta_6 \in \Irr(W_1)$ of degree $2$. As before, we only have to consider the non-linear characters further.

In order to do this, we first have to study the automorphism $\varphi_2$ induced by $(\gamma_2)^f$ on $W_1$. We know that $v=(n_\alpha(-1) n_\beta(-1))^3$ corresponds to the longest element of the Weyl group. Let $x=n_\alpha(-1)n_\beta(-1)n_\alpha(-1) \in V$. Then, we have $x \gamma^f(x) = v$ and $(x \gamma^f)^2=v F$. Therefore, $(\gamma_2)^f$ equals $x \gamma^f$ up to conjugation in $\mathbf{T}_0$. This shows that $\varphi_2$ acts on $W_1$ by ${s_\alpha s_\beta s_\alpha} \varphi $.

As before, we can now explicitly compute the character table of $W_1 \rtimes \langle {s_\alpha   s_\beta s_\alpha}  \varphi \rangle$ and identify the extensions of $\eta_5$ and $\eta_6$, see Table \ref{tabWchar}. We see that the character extensions of one of these are rational and the values of the extensions of the other character are contained in $\mathbb{Q}(\sqrt{-3})$. By Lemma \ref{lemmanumbertheory}(b), these extensions are all $\mathcal{H}$-invariant. 

Since field automorphisms act trivially on $W_1$, we can continue analogously to the proof of Lemma \ref{lemmad1} to obtain an $(\Aut(\G^{vF})_R \times \mathcal{H})_\psi$-invariant extension of $\psi$ to $\mathbf{N}^{vF} \rtimes (\Aut(\G^{vF})_R)_\psi$.
\end{proof}

\begin{table}
\begin{tabular}{|c|cccccc||ccc||ccc|}

\multicolumn{7}{c}{Non-linear characters of $W_1$}&
\multicolumn{3}{c}{Extensions to $W_1 \rtimes \langle \varphi \rangle$}&
\multicolumn{3}{c}{Extensions to $W_1 \rtimes \langle \varphi_2 \rangle$} \\ \hline
			& $()$&$s_\alpha$&$r$&$r^2$&$s_\beta$&$r^3$&$(s_\alpha r,\varphi)$ & $(s_\alpha,\varphi)$ & $(s_\alpha r^2, \varphi)$  &$(1,  \varphi _2)$ & $(r,  \varphi_2)$ & $(r^5, \varphi_2)$	\\ \hline
$\eta_5$	&$2$&$0$&$-1$&$-1$&$0$&$2$ 	&$2$&$-1$&$-1$ &$2$& $-1$ & $-1$\\
&&&&&&&$-2$&$1$&$1$ & $-2$& $1$ & $1$\\
$\eta_6$	&$2$&$0$&$1$&$-1$&$0$&$-2$ 	&$0$&$\sqrt{3}$&$-\sqrt{3}$ &$0$&$\sqrt{-3}$&$-\sqrt{-3}$\\
&&&&&&&$0$&$-\sqrt{3}$&$\sqrt{3}$ &$0$&$-\sqrt{-3}$&$\sqrt{-3}$\\ \hline
\end{tabular}
\caption{Extensions of the non-linear characters of $W_1$ to $W_1 \rtimes \langle \varphi \rangle$ and $W_1 \rtimes \langle \varphi_2 \rangle$. We write
$r=s_\beta s_\alpha$.}
\label{tabWchar}
\end{table}

\begin{lemma} \label{lemmad36}
Assume $d=3$ or $d=6$. Then, every $\psi \in \Irr_{\ell'}(N)$ has an $(\Aut(G)_R \times \mathcal{H})_\psi$-invariant extension to $N \rtimes (\Aut(G)_R)_\psi$.
\end{lemma}
\begin{proof}
Let $\psi$ be the character corresponding to the pair $(s, \eta) \in \mathcal{M}$. If $s \neq 1$, we can argue as before. 
If $s=1$, the irreducible characters of $W_3 \cong W_6  \cong C_{6}$ are linear and by Lemma \ref{lemmanumbertheory}(c) also $\mathcal{H}$-invariant. Therefore, the corresponding character $\psi$ is also linear and $\mathcal{H}$-invariant and can be trivially extended to $N \rtimes (\Aut(G)_R)_\psi$. This extension is naturally $(\Aut(G)_R \times \mathcal{H})_\psi$-invariant.
\end{proof}

We now turn to the global characters and their extensions. 
\begin{lemma} \label{lemmaglobal}
Every $\chi \in \Irr_{\ell'}(G)$ has a $(D \times \mathcal{H})_\chi$-invariant extension to $G \rtimes D_\chi$.
\end{lemma}
\begin{proof}
The full generic character table of $G$ is described in \cite{enomoto1976g23} and also contained in \cite{chevie}. If $\chi$ is a real character with odd degree, there is a unique real extension to $G \rtimes D_\chi$ by \cite[Theorem 2.3]{navarrotiep2008rational} that is thereby invariant under $(D \times \mathcal{H})_\chi$. If $\chi$ is semisimple or regular, we can use Gelfand--Graev characters and Deligne--Lusztig theory for disconnected reductive groups to construct a $(\langle F_3 \rangle \times \mathcal{H})_\chi$-invariant extension of $\chi$ to $G \rtimes D_\chi$, see \cite[Proposition 4.13]{johanssondiss}. The proof there also works if we consider the action of $(\langle \gamma \rangle \times \mathcal{H})_\chi$ on the constructed character extensions and we see that they are also $(D \times \mathcal{H})_\chi$-invariant.

In the notation of \cite{enomoto1976g23}, the remaining characters occurring in $\Irr_{\ell'}(G)$ for some $\ell$ are the unipotent characters $ \theta_1, \theta_2, \theta_{10}, \theta_{11}, \theta_{12}(1), \theta_{12}(-1)$. Let $\chi$ be one of these characters. The characters $\theta_{12}(\pm 1)$ only occur for $d=3$ and $d=6$ and are thereby $\mathcal{H}$-invariant by Lemma \ref{lemmanumbertheory}(c). The other characters are rational and all of them are $D$-invariant by \cite[Theorem 2.5]{malle2008extuni}. 

We consider the Deligne--Lusztig characters $R_{\mathbf{S}_3}^{\G}(1)$ and $R_{\mathbf{S}_6}^{\G}(1)$. As we see in Table \ref{tabless}, any non-trivial semisimple $s \in \mathbf{S}^F_3$ satisfies $C_{\G}(s)=\mathbf{S}_3$ and is thereby regular.
Using \cite{chevie}, we can read off the character values $\rho(s)$ for a regular element $s \in \mathbf{S}^F_3$ and any unipotent character $\rho$ and thereby determine the multiplicity $\langle R_{\mathbf{S}_3}^{\G}(1), \rho \rangle =\rho(s)$, see for example \cite[Remark 2.3.10]{geckmalle2020}. Analogously, this can be done for $R_{\mathbf{S}_6}^{\G}(1)$ and we obtain 
$$R_{\mathbf{S}_3}^{\G}(1)=\theta_0 - \theta_2 + \theta_{10} - \theta_{12}(1) -\theta_{12}(-1)+  \theta_5, \quad R_{\mathbf{S}_6}^{\G}(1)=\theta_0 - \theta_1 + \theta_{11} +\theta_{12}(1) +\theta_{12}(-1)+  \theta_5 $$
in the notation of \cite{enomoto1976g23}. 

We can now argue as in the proof of \cite[Proposition 6.12]{johanssondiss} to obtain an $\mathcal{H}_\rho$-invariant extension to $G \rtimes \langle \gamma^2 \rangle$ for every irreducible constituent $\rho$ of $R_{\mathbf{S}_3}^{\G}(1)$ and $R_{\mathbf{S}_6}^{\G}(1)$. Note that this is possible since the considered Deligne--Lusztig characters are multiplicity-free, $\gamma^2$ acts trivially on the relative Weyl groups, and $\G$ has connected center. Other than that, the proof of \cite[Proposition 6.12]{johanssondiss} does not use any assumptions on $\G$ and $F$. This yields an $\mathcal{H}$-invariant extension $\widehat{\chi} \in \Irr(G \rtimes \langle \gamma^2 \rangle)$ of $\chi$.

In \cite{brunat2007extensiong2graph}, Brunat computed the extensions of the $\gamma^f$-stable irreducible characters of $G$ to $G \rtimes \langle \gamma^f \rangle$. We know that $\chi$ extends to some $\widetilde{\chi} \in \Irr(G \rtimes \langle \gamma^f \rangle)$. 
We see that $\widetilde{\theta_1}$ and $\widetilde{\theta}_{11}$ are rational, the values of $\widetilde{\theta}_2$ are contained in $\mathbb{Q}(\sqrt{3})$, and the values of $\widetilde{\theta}_{10}$, $\widetilde{\theta}_{12}(1)$, $\widetilde{\theta}_{12}(-1)$ are contained in $\mathbb{Q}(\sqrt{-3})$. Thus, they are all $\mathcal{H}$-invariant by Lemma \ref{lemmanumbertheory}.

The unique common extension of $\widehat{\chi}$ and $\widetilde{\chi}$ to $G \rtimes D_\chi$ is invariant under $(D\times \mathcal{H})_{\chi}=D_\chi \times \mathcal{H}$. This shows the claim.
\end{proof}
This means that, contrary to our motivation to study $G$ described above, we find $\mathcal{H}$-invariant extensions of all local and global characters if $f$ is odd.
 
With this, we can now prove Theorem \ref{theoremA} for type $\mathsf{G}_2$.
\begin{theorem}
The inductive McKay--Navarro condition holds for the group $\mathsf{G}_2(3^f)$ and all primes $\ell$ if $f$ is an odd integer. 
\end{theorem}
\begin{proof}
For $\ell=2$ and $\ell=3$ we already know that the statement holds by \cite{johansson2020} and \cite{johansson2021}. The group $3.\mathsf{G}_2(3)$ can be considered with \cite{GAP4} and we see that Condition \ref{cond} holds.

If $f \geq 2$, the equivariance condition is satisfied by Proposition \ref{propeq}. For odd $f$, the extension condition is satisfied by Lemma \ref{lemmad1}, Lemma \ref{lemmad2}, Lemma \ref{lemmad36}, and Lemma \ref{lemmaglobal} since we can extend the constructed character extensions canonically to the inner automorphisms in $\Aut(G)_R$. This shows the claim.
\end{proof}

\begin{rem}
The construction of the extensions does not work in the same way if $f$ is even. 
For some of the local and global characters of $X \in \{G,N\}$ we constructed extensions to $X \rtimes \langle F_3 \rangle$ and $X \rtimes \langle \gamma^f \rangle$ separately. If $f$ is odd, $\gamma^f$ is a field automorphism and we have to study the extensions to $X \rtimes \langle \gamma^{f_2'} \rangle$ and $X \rtimes \langle F_3^{f_2} \rangle$ where $f_2 f_2'=f$ such that $f_2'$ is odd and $f_2$ is even. In particular, we cannot use the results in \cite{brunat2007extensiong2graph} directly but would have to imitate the calculations from there for extensions of degree $2f_2$. 

Further, we would like to point out that for some extensions of the global characters we used Deligne--Lusztig characters for disconnected groups. Since we need here that the automorphism extending $G$ acts trivially on the Weyl group, we cannot use these methods to extend the characters to the exceptional graph automorphism. The same problem occurs using Harish--Chandra induction for disconnected groups as in \cite{ruhstorfer2021inductive}.
\end{rem}

\section{The groups $\mathsf{B}_2(2^f)$}
In this section, we prove the analogous statement for groups of type $\mathsf{B}_2$ over a field of characteristic $2$. This can be done analogously to the considerations for $\mathsf{G}_2(3^f)$.

From now on, let $\G$ be a simple algebraic group of type $\mathsf{B}_2$ over an algebraically closed field of characteristic $2$. We denote the standard Frobenius map of $\G$ by $F_2$ and set $F:=F_2^f$, $q:=2^f$ for some integer $f \geq 2$. As before, $G:=\GF=\mathsf{B}_2(2^f)$ is a simple group with trivial Schur multiplier and the outer automorphism group of $G$ is $D=\langle \gamma \rangle$ where $\gamma$ is the exceptional graph automorphism of $\G$.
Note that $\G$ has trivial center and we have a bijection between rational semisimple elements of $G$ and its dual with an isomorphism of centralizers \cite[ p. 164]{lusztig1977}.

As before, we consider the inductive McKay--Navarro condition for a prime $\ell \geq 3$ dividing 
$|G|=q^4 \phi_1^2 \phi_2^2 \phi_4$.
Let $d$ be the order of $q$ modulo $\ell$, i.e. $\ell \mid \phi_d$. 
We continue to use the symbols we introduced in Section \ref{sectiong2}. 
Note that all Sylow $d$-tori are again regular and therefore also maximal tori. 

The character parametrizations described in Section \ref{sectioneq} also exist for $\GF$ and $N$. The relevant information about the centralizers of semisimple elements is displayed in Table \ref{tablessB}. Again, this can be used to prove the equivariance part of the inductive condition.

\begin{table}[tbh]
\begin{tabular}{cccccc}
\hline $d$ & Conjugacy Class & Type of $\mathbf{H}$& $|\mathbf{H}^F|$&$|s^N|$ &$\mathbf{W}_{\mathbf{H}}(\mathbf{S}_d)^F$ \\ \hline
$1$ & $A_1 =1 $ & $\mathsf{B}_2(q)$ & $q^4(q^2-1)(q^4-1)$ & $1$ & $D_{8}$ \\ 
 & $B_1(i,j) $ & $\mathsf{A}_0(q) \cdot  T(\phi_1^2)$ & $(q-1)^2$ & $8$ & $1$ \\
& $C_{1}(i) $ & $\mathsf{A}_1(q) \cdot T(\phi_1)$ & $q(q-1)^2(q+1)$ & $4$ & $C_2 $ \\
&$C_{2}(i) $ & $\mathsf{A}_1(q) \cdot T(\phi_1)$ & $q(q-1)^2(q+1)$ & $4$ & $C_2 $  \\ \hline 

$2$ & $A_1 =1 $ & $\mathsf{B}_2(q)$ & $q^4(q^2-1)(q^4-1)$ & $1$ & $D_{8}$ \\ 
 & $B_4(i,j) $ & $\mathsf{A}_0(q) \cdot  T(\phi_2^2)$ & $(q+1)^2$ & $8$ & $1$ \\
& $C_{3}(i) $ & $\mathsf{A}_1(q) \cdot T(\phi_2)$ & $q(q+1)^2(q-1)$ & $4$ & $C_2 $ \\
&$C_{4}(i) $ & $\mathsf{A}_1(q) \cdot T(\phi_2)$ & $q(q+1)^2(q-1)$ & $4$ & $C_2 $  \\ \hline 

$4$ & $A_1 =1 $ & $\mathsf{B}_2(q)$ & $q^4(q^2-1)(q^4-1)$ & $1$ & $C_{4}$ \\ 
 & $B_5(i) $ & $\mathsf{A}_0(q) \cdot T(\phi_4)$ & $(q^2+1)$ & $4$ & $1$ \\ \hline 

\end{tabular}
\caption{Conjugacy classes of semisimple elements $s \in \mathbf{S}_d^F$ using the notation of \cite[Table IV-1]{enomoto1972sp4} and description of the structure of their centralizers $\mathbf{H}=C_{\G}(s)$.}
\label{tablessB}
\end{table}

\begin{lemma} \label{lemmaeqB2}
There is a $D \times \mathcal{H}$-equivariant bijection $\Irr_{\ell'}(G) \rightarrow \Irr_{\ell'}(N)$. 
\end{lemma}
\begin{proof}
Since $q$ is even, $H$ is trivial, see \cite[Setting 2.1]{spaeth2009mckayexceptional}. We can therefore choose the trivial extension map for $H \triangleleft V$ and use this to construct an $ND \times \mathcal{G}$-equivariant extension map for $\mathbf{S}_d^F \triangleleft N$ as before. The characters of the occurring relative Weyl groups are again rational or, if $d=4$, contained in $\mathbb{Q}(\zeta_4)$ where $\zeta_4$ is a fourth primitive root of unity. By Lemma \ref{lemmanumbertheory}(e), they are $\mathcal{H}$-invariant. This already shows that the parametrization of the local characters is $D \times \mathcal{G}$-equivariant.

For the global characters, we can argue as before since the unipotent characters of groups of type $\mathsf{A}_1$ and $\mathsf{B}_2$ are again rational. 
\end{proof}

We now restrict ourselves to the case that $f$ is odd and consider character extensions.

\begin{lemma} \label{lemmaBlocal}
Every $\psi \in \Irr_{\ell'}(N)$ has an $(\Aut(G)_R \times \mathcal{H})_\psi$-invariant extension to $N \rtimes (\Aut(G)_R)_\psi$.
\end{lemma}
\begin{proof}
If $\psi$ is a linear character, we can argue as in Section \ref{sectionext}. Otherwise, we have $d =1$ or $d=2$ and $\psi$ corresponds to the pair $(1,\eta)$ for $\eta \in \Irr(\mathbf{W}_d)$ of degree $2$. If $d=1$, the extension of $\eta$ to $W_1 \rtimes \langle \varphi \rangle$ where $\varphi \in \Aut(W_1)$ is induced by $\gamma$ is contained in $\mathbb{Q}(\sqrt{2})$. We further see computationally that this is the case for every conjugate of $\varphi$ that has order $2$. Using Lemma \ref{lemmanumbertheory}(d) and arguing as before, this shows the claim.
\end{proof}

\begin{lemma} \label{lemmaBglobal}
Every $\chi \in \Irr_{\ell'}(G)$ has a $(D \times \mathcal{H})_\chi$-invariant extension to $G \rtimes D_\chi$.
\end{lemma}
\begin{proof}
The generic character table of $G$ is again contained in \cite{chevie}. If $\chi$ is semisimple or regular, we can argue as in the proof of Lemma \ref{lemmaglobal} to find a suitable extension. Otherwise, $\chi$ is unipotent and again a constituent of a multiplicity-free Deligne--Lusztig character. As before, we find an $\mathcal{H}$-invariant extension of $\chi$ to $G \rtimes \langle F_3 \rangle_\chi$. 

The action of $\gamma$ interchanges two of these unipotent characters. Thus, we only have to consider the extensions of $\theta_1, \theta_5$ (in the notation of \cite{enomoto1972sp4}) to $G \rtimes \langle \gamma^f \rangle$. They have been computed in \cite{brunat2006extSuz} and we see that both are rational. Since all unipotent characters are rational, this already shows the claim.
\end{proof}

We can now prove Theorem \ref{theoremA} for type $\mathsf{B}_2$.
\begin{theorem}
The inductive McKay--Navarro condition holds for the group $\mathsf{B}_2(2^f)$ and all primes $\ell$ if $f$ is an odd integer.
\end{theorem}
\begin{proof}
For $\ell=2$ we already know that the statement holds by \cite{johansson2020}. The group $\mathsf{B}_2(2)'$ can be considered with \cite{GAP4} and we see that the extension part of Condition \ref{cond} is not always satisfied since we do not always have invariant character extensions. However, the action of $\mathcal{H}$ on the local and global character extensions is equivariant and thereby \cite[Definition 3.5]{navarro2019reduction} holds.

If $f \geq 2$, the equivariance condition is satisfied by Proposition \ref{lemmaeqB2}, and, if $ f$ is odd, the extension condition is satisfied by Lemma \ref{lemmaBlocal} and \ref{lemmaBglobal}.
\end{proof}

\subsection*{Acknowledgement} 
I would like to thank Gunter Malle for his suggestions and comments. This work was financially supported by the SFB-TRR 195 of the German Research Foundation (DFG).

\bibliography{biblio}

\newcommand{\etalchar}[1]{$^{#1}$}
\providecommand{\bysame}{\leavevmode\hbox to3em{\hrulefill}\thinspace}
\providecommand{\MR}{\relax\ifhmode\unskip\space\fi MR }
\providecommand{\MRhref}[2]{%
  \href{http://www.ams.org/mathscinet-getitem?mr=#1}{#2}
}
\providecommand{\href}[2]{#2}
\begin{thebibliography}{{J}oh22a}

\bibitem[{B}ru06]{brunat2006extSuz}
O.~{B}runat, \emph{The {S}hintani descents of {S}uzuki groups and their
  consequences}, J. Algebra \textbf{303} (2006), 869--890.

\bibitem[Bru07]{brunat2007extensiong2graph}
O.~Brunat, \emph{{On the extension of {$G_2(3^{2n+1})$} by the exceptional
  graph automorphism}}, Osaka J. Math. \textbf{44} (2007), no.~4, 973--1023.

\bibitem[CS13]{cabanesspaeth2013}
M.~Cabanes and B.~Sp{\"a}th, \emph{Equivariance and extendibility in finite
  reductive groups with connected center}, Math. Z. \textbf{275} (2013),
  689--713.

\bibitem[{E}no72]{enomoto1972sp4}
H.~{E}nomoto, \emph{The characters of the finite symplectic group {${Sp}(4,q)$,
  $q=2^f$}}, Osaka J. Math. \textbf{9} (1972), no.~1, 75--94.

\bibitem[Eno76]{enomoto1976g23}
H.~Enomoto, \emph{The characters of the finite {C}hevalley group {${G}_2(q)$,
  $q=3^f$}}, Japan. J. Math. \textbf{2} (1976), no.~2, 191--248.

\bibitem[GAP19]{GAP4}
The GAP~Group, \emph{{GAP -- Groups, Algorithms, and Programming, Version
  4.10.2}}, 2019,
  \href{https://www.gap-system.org}{\url{{https://www.gap-system.org}}}.

\bibitem[GHL{\etalchar{+}}96]{chevie}
M.~Geck, G.~Hiss, F.~L{\"u}beck, G.~Malle, and G.~Pfeiffer,
  \emph{\textsf{CHEVIE} -- {A} system for computing and processing generic
  character tables for finite groups of {L}ie type, {W}eyl groups and {H}ecke
  algebras}, Appl. Algebra Engrg. Comm. Comput. \textbf{7} (1996), 175--210.

\bibitem[GM20]{geckmalle2020}
M.~Geck and G.~Malle, \emph{The {C}haracter {T}heory of {F}inite {G}roups of
  {L}ie {T}ype: {A} {G}uided {T}our}, Cambridge Studies in Advanced
  Mathematics, vol. 187, Cambridge University Press, Cambridge, 2020.

\bibitem[GLS98]{gorensteinlyonssolomon1998}
D.~Gorenstein, R.~Lyons, and R.~Solomon, \emph{The {C}lassification of the
  {F}inite {S}imple {G}roups, {N}umber 3}, Mathematical Surveys and Monographs,
  vol.~40, Amer. Math. Soc., 1998.

\bibitem[{J}oh21]{johansson2021}
{B}. {J}ohansson, \emph{The inductive {McKay}--{N}avarro condition for the
  {S}uzuki and {R}ee groups}, 2021,
  \href{https://arxiv.org/abs/2110.10946}{arXiv:2110.10946}.

\bibitem[{J}oh22a]{johanssondiss}
B.~{J}ohansson, \emph{The inductive {McKay}--{N}avarro condition for finite
  groups of {L}ie type}, Dissertation, Technische Universit{\"a}t
  Kaisers\-lautern, 2022.

\bibitem[Joh22b]{johansson2020}
B.~Johansson, \emph{On the inductive {McKay}--{N}avarro condition for finite
  groups of {L}ie type in their defining characteristic}, J. Algebra
  \textbf{610} (2022), 223--240.

\bibitem[Kle88]{kleidman1988maxsubgroups}
P.~Kleidman, \emph{The maximal subgroups of the {C}hevalley groups {$G_2(q)$}
  with q odd, the {R}ee groups {$^2G_2(q)$}, and their automorphism groups}, J.
  Algebra \textbf{117} (1988), 30--71.

\bibitem[Lus77]{lusztig1977}
G.~Lusztig, \emph{Irreducible representations of finite classical groups},
  Invent. math. \textbf{43} (1977), 125--175.

\bibitem[Lü]{luebeckhomepagecent}
F.~Lübeck, \emph{Centralizers and numbers of semisimple classes in exceptional
  groups of {L}ie type},
  \href{https://www.math.rwth-aachen.de/~Frank.Luebeck/chev/CentSSClasses/}{\url{https://www.math.rwth-aachen.de/~Frank.Luebeck/chev/CentSSClasses/}}.

\bibitem[Mal07]{malle2008height0}
G.~Malle, \emph{Height 0 characters of finite groups of {L}ie type}, Represent.
  Theory \textbf{11} (2007), 192--220.

\bibitem[Mal08]{malle2008extuni}
G.~Malle, \emph{Extensions of unipotent characters and the inductive {McKay}
  condition}, J. Algebra \textbf{320} (2008), 2963--2980.

\bibitem[MS15]{mallespaeth2015mckayfor2}
G.~Malle and B.~Sp{\"a}th, \emph{Characters of odd degree}, Ann. of Math.
  \textbf{184} (2015), 869--908.

\bibitem[{N}av04]{navarro2004mckayrefinement}
G.~{N}avarro, \emph{The {M}c{K}ay conjecture and {G}alois automorphisms}, Ann.
  of Math. \textbf{160} (2004), no.~3, 1129--1140.

\bibitem[NSV20]{navarro2019reduction}
G.~Navarro, B.~Sp{\"a}th, and C.~Vallejo, \emph{A reduction theorem for the
  {G}alois--{M}c{K}ay conjecture}, Trans. Amer. Math. Soc. \textbf{373} (2020),
  6157--6183.

\bibitem[NT08]{navarrotiep2008rational}
G.~Navarro and P.~H. Tiep, \emph{Rational {i}rreducible {c}haracters and
  {r}ational {c}onjugacy {c}lasses in {f}inite {g}roups}, Trans. Amer. Math.
  Soc. \textbf{360} (2008), no.~5, 2443--2465.

\bibitem[Ruh21]{ruhstorfer2017navarro}
L.~Ruhstorfer, \emph{The {N}avarro refinement of the {McKay} conjecture for
  finite groups of {Lie} type in defining characteristic}, {J}. {A}lgebra
  \textbf{582} (2021), 177--205.

\bibitem[RSF22a]{ruhstorfer2021inductive}
L.~Ruhstorfer and A.~A. Schaeffer~Fry, \emph{The inductive {M}c{K}ay--{N}avarro
  conditions for the prime 2 and some groups of {L}ie type}, Proc. Amer. Math.
  Soc. Ser. B \textbf{9} (2022), 204--220.

\bibitem[RSF22b]{ruhstorferSF2022}
L.~{R}uhstorfer and A.~A. Schaeffer~Fry, \emph{{N}avarro's {G}alois--{McKay}
  conjecture for the prime $2$}, 2022,
  \href{https://arxiv.org/abs/2211.14237}{arXiv:2211.14237}.

\bibitem[SF22]{fry2020galoisequivariant}
A.~A. Schaeffer~Fry, \emph{Galois-equivariant {McKay} bijections for primes
  dividing $q-1$}, Israel J. Math. \textbf{247} (2022), 269--302.

\bibitem[Sp{\"a}09]{spaeth2009mckayexceptional}
B.~Sp{\"a}th, \emph{The {McKay} conjecture for exceptional groups and odd
  primes}, Math. Z. \textbf{261} (2009), 571--595.

\bibitem[Spr74]{springer74}
T.~A. Springer, \emph{Regular elements of finite reflection groups}, Invent.
  math. \textbf{25} (1974), 159--198.

\bibitem[SV20]{srinivasanvinroot2019}
B.~Srinivasan and C.R. Vinroot, \emph{Galois group action and {J}ordan
  decomposition of characters of finite reductive groups with connected
  center}, J. {A}lgebra \textbf{558} (2020), 708--727.

\end{thebibliography}
\bibliographystyle{amsalpha2} 
\end{document}